\documentclass[12pt]{amsart}
\usepackage{amsmath}
\usepackage{amscd}
\usepackage{amssymb}
\usepackage{amsfonts}

\newtheorem{theorem}{Theorem}[section]
\newtheorem{lemma}[theorem]{Lemma}

\newtheorem{proposition}[theorem]{Proposition}

\theoremstyle{definition}

\theoremstyle{remark}
\newtheorem{remark}[theorem]{Remark}
\numberwithin{equation}{section}

\begin{document}

\title[Diameter estimate for closed manifolds]
{Diameter estimate for closed manifolds with positive scalar curvature}

\author{Xuenan Fu}
\address{Department of Mathematics, Shanghai University, Shanghai 200444, China}
\email{xuenanfu97@163.com}

\author{Jia-Yong Wu}
\address{Department of Mathematics, Shanghai University, Shanghai 200444, China}
\email{wujiayong@shu.edu.cn}

\thanks{}
\dedicatory{}
\date{\today}


\begin{abstract}
For a simply connected closed Riemannian manifold with positive scalar curvature,
we prove an upper diameter bound in terms of its scalar curvature integral, the
Yamabe constant and the dimension of the manifold. When a manifold has a
conformal immersion into a sphere, the dependency on the Yamabe constant is not
necessary. The power of scalar curvature integral in these diameter estimates
is sharp and it occurs at round spheres with canonical metric.
\end{abstract}
\maketitle

\section{Introduction}\label{Int1}
One of central problems in Riemannian geometry is to investigate the relationship
between curvature and topology of Riemannian manifolds. The classical Myers' theorem
\cite{[M]} states that if the Ricci curvature of an $n$-dimensional complete connected
Riemannian manifold $(M,g)$ satisfies $\mathrm{Ric}\ge(n-1)K$ for some constant $K>0$,
then $(M,g)$ is compact and its diameter is at most \(\pi/\sqrt{K}\). Here the diameter
of $(M,g)$ is defined by
\[
\mathrm{diam}(M):=\max\left\{dist(x,y)|\,\,\,\forall\,x,\, y\in M\right\},
\]
where $dist(x,y)$ denotes the geodesic distance from point $x$ to point $y$. A natural
question about the Myers' result is that if one can get an upper diameter estimate
under some scalar curvature assumption instead of the Ricci curvature condition. In
this paper, we will study this question and obtain an upper diameter estimate for
any closed (i.e. compact without boundary) manifold with positive scalar curvature.
Our diameter upper estimate depends on the $L^{\frac{n-1}{2}}$-norm of scalar curvature,
the Yamabe constant and the dimension of the manifold. In particular, if a closed
manifold has a conformal immersion into $n$-sphere, the dependency on the Yamabe
constant is not necessary.

To state the main result, we start to recall some basic facts about the Yamabe constant.
For a closed differential manifold $M$ of dimension $n\ge 3$, the normalized Einstein-Hilbert
functional $\mathcal {E}$ assigning to each Riemannian metric $g$ is defined by
\[
\mathcal {E}(g)=\frac{\int_M\mathrm{R}_gdv_g}{\left(\int_Mdv_g\right)^{\frac{n-2}{n}}},
\]
where $\mathrm{R}_g$ is the scalar curvature of $(M,g)$ and $dv_g$ is the volume element
associated to the metric $g$. This functional is scale-invariant and can be regarded as
measuring the average scalar curvature of metric $g$ over $M$.  It was conjectured by
Yamabe that every conformal class on any smooth closed manifold contains a metric of
constant scalar curvature (the so-called Yamabe problem). This problem was proved by
Yamabe \cite{[Y]}, Trudinger \cite{[T]}, Aubin \cite{[A]} and Schoen \cite{[S]} that
a minimum value of $\mathcal {E}$ is attained in each conformal class of metrics,
and this minimum is achieved by a metric of constant scalar curvature. In particular,
each conformal class $[g]$ of $g$ has an associated \emph{Yamabe constant} $Y(M,[g])$,
given by
\begin{equation}\label{YIdef}
\begin{aligned}
Y(M,[g]):=&\inf_{0<\varphi\in C^\infty(M)}\left\{\mathcal {E}(\tilde{g})\big|\tilde{g}=\varphi^{\frac{4}{n-2}}g\in[g]\right\}\\
=&\inf_{0<\varphi\in C^\infty(M)}\frac{\frac{4(n-1)}{n-2}\int_M|\nabla\varphi|^2dv_g+\int_M \mathrm{R}\,\varphi^2dv_g}
{\left(\int_M\varphi^{\frac{2n}{n-2}}dv_g\right)^{\frac{n-2}{n}}}.
\end{aligned}
\end{equation}
The above infimum is finite and not to be negative infinite. Moreover, $Y(M,[g])>0$
on a closed Riemannian manifold $(M,g)$ if and only if the conformal class $[g]$ contains a
conformal metric with positive scalar curvature everywhere.  For more results about the Yamabe
problem, the interested reader are referred to surveys \cite{[LP],[Ak]}.

The main result of this paper is that

\begin{theorem}\label{Main}
Let $(M,g)$ be an $n$-dimensional $(n\ge3)$ simply connected closed manifold with
the scalar curvature $\mathrm{R}>0$. There exists a constant $C(n,Y)$ depending
on $n$ and the Yamabe constant $Y:=Y(M,[g])$ such that
\[
\mathrm{diam}(M)\le C(n,Y)\int_M\mathrm{R}^{\frac{n-1}{2}}dv.
\]
In particular, we can take
\[
C(n,Y)=4\max\left\{w^{-1}_n, \left(2en^{-1}Y\right)^{-\frac n2}e^{\frac{17n-18}{n-2}\cdot2^n}\right\},
\]
where $w_n$ is the volume of the unit $n$-sphere $\mathbb{S}^n$.
\end{theorem}

\begin{remark}
A different expression (but the same in spirit) of the diameter upper estimate 
was proved by K. Akutagawa \cite{[Ak2]} by using a different approach. Our estimate 
form may be likely to reflect the dependence on the energy of scalar 
curvature.
\end{remark}

The exponent $\frac{n-1}{2}$ of scalar curvature integral in Theorem \ref{Main}
is sharp, which can be achieved by the $n$-sphere $\mathbb{S}^n(r)$ of large
radius $r$ with the canonical metrics $g_0$. Indeed, in this special case, its
diameter is equivalent to $r$, i.e., $\mathrm{diam}_{g_0}(M)\approx r$, while
the right hand side of our estimate essentially equals to $c(n)r$ because
$\mathrm{R}(g_0)\approx r^{-2}$, $\mathrm{Vol}(\mathbb{S}^n(r))\approx r^n$
and $Y(M,[g])\approx c(n)$. The coefficient $C(n,Y)$ is of course not optimal,
which might be sharpened by choosing a delicate cut-off function in Theorem
\ref{logeq} below.

For a $4$-dimensional closed simply connected Riemannian manifold $(M,g)$, the
Yamabe constant $Y(M,[g])$ could be replaced by some other curvature integrals
via the following Gursky's formula (see \cite{[G]})
\[
\int_M\mathrm{R}^2dv-12\int_M|\overset{\circ}{\mathrm{Ric}}|^2dv\le Y^2(M,[g]),
\]
where $\overset{\circ}{\mathrm{Ric}}$ denotes the traceless part of the Ricci tensor.

Zhang \cite{[Zhq]} proved a diameter estimate depending on the $L^{\frac{n-1}{2}}$-norm
of the scalar curvature, the volume of manifold and the positive Yamabe constant. Deng
\cite{[De]} also proved a similar diameter bound depending on the $L^{n-1}$-norm of the scalar
curvature, the volume of manifold and the positive Yamabe constant. But our result
indicates that the dependency on the volume of manifold is not necessary.

Recall that if $(M,g)$ and $(\widetilde{M},\widetilde{g})$ are Riemannian manifolds,
an immersion $\Psi:\,(M,g)\to(\widetilde{M},\widetilde{g})$ is said to be conformal
if there exists $f\in C^\infty(M)$ such that
\[
\Psi^{*}\widetilde{g}=e^f g.
\]
Below we will see that if a closed manifold has a conformal immersion into
$n$-sphere, then the diameter estimate does not depend on the Yamabe constant.
\begin{theorem}\label{Mainimm}
Let $(M,g)$ be an $n$-dimensional $(n\ge3)$ simply connected closed manifold.
Assume that there exists a conformal immersion
\[
\Psi:\,(M,g)\to (\mathbb{S}^n,g_0),
\]
where $(\mathbb{S}^n,g_0)$ denotes the standard unit sphere in $\mathbb{R}^{n+1}$.
There exists a constant $C(n)$ depending only on $n$ such that
\[
\mathrm{diam}(M)\le C(n)\int_M\mathrm{R}_{+}^{\frac{n-1}{2}}dv,
\]
where $\mathrm{R}_{+}$ denotes the positive part of the scalar
curvature $\mathrm{R}$.
\end{theorem}
\begin{remark}

On any simply connected conformally flat Riemannian manifold, there exists
a conformally immersion from $(M,g)$ to $(\mathbb{S}^n,g_0)$; see for instance
the explanation in \cite{[KP]}. Hence the assumption of Theorem \ref{Mainimm}
includes the conformally flat manifold as a special case.
\end{remark}

Our proof is motivated by the argument of Topping's papers \cite{[To2],[To3]}.
Recall that Topping \cite{[To2]} used the Perelman's $\mathcal{W}$-functional
to prove an upper diameter bound for a closed manifold in terms of scalar curvature
integral under the Ricci flow. In \cite{[To3]}, Topping applied the Michael-Simon
Sobolev inequality to get an upper diameter estimate for a closed connected manifold
immersed in the Euclidean space in terms of its mean curvature integral. In our
setting, we first apply the Yamabe constant to get the Yamabe-Sobolev inequality
and furthermore get a logarithmic Sobolev inequality on closed manifolds with
positive scalar curvature. Then we use the logarithmic Sobolev inequality and
cut-off functions to prove a new functional inequality. This functional inequality
relates a maximal function of scalar curvature and the volume ratio (see Theorem
\ref{logeq}). Later, we apply the functional inequality to prove an alternative
theorem, which states that the maximal function and the volume ratio cannot be
simultaneously smaller than a fixed constant on a geodesic ball (see Theorem
\ref{maxirati}). Finally, we use the alternative theorem and a Vitali-type covering
lemma to prove the diameter estimate. When a closed manifold has a conformal
immersion into a sphere, we can utilize another Sobolev inequality (see Proposition
\ref{immerpro}) and follow the above procedure to prove Theorem \ref{Mainimm}.

In the past few years, Topping's results have been generalized by Zheng and the second
author \cite{[WZ]}, Zhang \cite{[Zhq]}, Deng \cite{[De]} and the second author \cite{[Wu]},
etc.. In fact, in \cite{[WZ]}, Zheng and the second author proved an upper
diameter bound for a closed manifold immersed in the ambient manifold in terms of
its mean curvature integral. In \cite{[Zhq]}, Zhang applied the uniform Sobolev
inequality along the Ricci flow to obtain an upper diameter bound depending only
on the $L^{(n-1)/2}$ bound of the scalar curvature, volume and the Sobolev constant
(or the Yamabe constant) under the Ricci flow. In \cite{[De]}, Deng applied the
Yamabe-Sobolev inequality to detect the compactness of a class of complete
manifolds and also proved an upper diameter estimate for such manifolds. Recently,
the second author \cite{[Wu]} applied the Perelman's entropy functional to prove a
sharp upper diameter bound for a compact shrinking Ricci soliton. In this direction,
further development can be referred to \cite{[AdM],[Pa],[Pa2], [Po]} and references
therein. Besides, we would like to mention that Bakry and Ledoux \cite{[BL]} applied
a sharp Sobolev inequality to give an alternative proof of the Myers' diameter estimate.

The structure of this paper is as follows. In Section \ref{sec2}, we give a
(logarithmic) Yamabe-Sobolev inequality on closed manifolds with positive scalar
curvature. By a suitable cut-off function, we reduce the Sobolev inequality
into a new functional inequality. In Section \ref{sec3}, we apply the
functional inequality to give an alternative lower bound between the maximal
function of scalar curvature and volume ratio. In Section \ref{sec4}, we apply
the alternative theorem to prove Theorem \ref{Main}. In Section \ref{sec5},
we adopt the same argument of Theorem \ref{Main} to prove Theorem \ref{Mainimm}.

\textbf{Acknowledgement}.
The authors thank Professor Kazuo Akutagawa for making them aware
of the work of his work \cite{[Ak2]}.

%55555555555555555555555555555555555555555555555555555555555555555555555555555555555555555555555555555555555555555555555555555555555555555555555555555555555555
%55555555555555555555555555555555555555555555555555555555555555555555555555555555555555555555555555555555555555555555555555555555555555555555555555555555555555
%55555555555555555555555555555555555555555555555555555555555555555555555555555555555555555555555555555555555555555555555555555555555555555555555555555555555555

\section{Functional inequalities}\label{sec2}
In this section, we will discuss some functional inequalities on closed manifolds
with positive scalar curvature. We first apply the Yamabe constant to get a
logarithmic Yamabe-Sobolev inequality on closed manifolds with positive scalar
curvature. Then we apply the Sobolev inequality to prove a new functional
inequality, which will be used in the proof of Theorem \ref{Main}.

On a closed manifold $(M,g)$, we have a fact that $[g]$ contains a conformal
metric with positive scalar curvature everywhere if and only if the Yamabe
constant is positive. So if scalar curvature $\mathrm{R}>0$
on $(M,g)$, then
\[
Y(M,[g])>0.
\]
Therefore the definition \eqref{YIdef} yields the Yamabe-Sobolev inequality
\begin{equation}\label{YSineq}
\left(\int_M\varphi^{\frac{2n}{n-2}}dv\right)^{\frac{n-2}{n}}
\le Y^{-1}(M,[g])\left(\int_M\frac{4(n-1)}{n-2}|\nabla \varphi|^2dv+\int_M\mathrm{R}\varphi^2dv\right)
\end{equation}
for any positive $\varphi\in C^\infty(M)$. The Yamabe-Sobolev type inequality
implies much geometric information, such as first eigenvalue estimates, gap
theorems, etc.. We refer the interested reader to \cite{[He]} and references
therein. Here we will see that \eqref{YSineq} immediatedly implies a
logarithmic Yamabe-Sobolev inequality.
\begin{lemma}\label{logsi}
Let $(M,g)$ be an $n$-dimensional $(n\ge3)$ closed manifold with positive scalar curvature.
For any positive $\varphi\in C^\infty(M)$ with
\[
\int_M\varphi^2dv=1
\]
and any real number $\tau>0$,
\begin{equation}\label{LSI}
\frac n2\ln\frac{2eY(M,[g])}{n}
\le\tau\int_M\left(\frac{4(n-1)}{n-2}|\nabla\varphi|^2+\mathrm{R}\varphi^2\right)dv
-\int_M\varphi^2\ln\varphi^2dv-\frac n2\ln\tau,
\end{equation}
where $\mathrm{R}$ is the scalar curvature of $(M,g)$ and $Y(M,[g])$ is the Yamabe constant.
\end{lemma}

\begin{proof}[Proof of Lemma \ref{logsi}]
Assume that the Yamabe-Sobolev inequality \eqref{YSineq} holds on a closed manifold $(M,g)$.
For any smooth function $\varphi$ with $\|\varphi\|_2=1$, we consider the weighted measure
$d\mu=\varphi^2dv$ on $(M,g)$, and then
\[
\int_M d\mu=1.
\]
Note that smooth function $\ln\Phi$ is concave with respect to positive parameter
$\Phi$. Applying the standard Jensen inequality
\[
\int_M\ln\Phi d\mu\le\ln\left(\int_M\Phi d\mu\right)
\]
to positive function $\Phi=\varphi^{q-2}$, where $q=\frac{2n}{n-2}$, we get that
\begin{align*}
\int_M(\ln \varphi^{q-2})\varphi^2dv&\le\ln\left(\int_M\varphi^{q-2}\varphi^2dv\right)\\
&=\ln \parallel\varphi\parallel^q_q.
\end{align*}
In other words,
\begin{align*}
\int_M\varphi^2\ln\varphi dv&\le\frac{q}{q-2}\ln \parallel\varphi\parallel_q\\
&=\frac n2\ln \parallel\varphi\parallel_q.
\end{align*}
Combining this with \eqref{YSineq}, we have the following estimate:
\begin{align*}
\int_M\varphi^2\ln\varphi^2dv&\le\frac n2\ln \parallel\varphi\parallel^2_q\\
&\le\frac n2\ln\left[Y^{-1}(M,[g])\left(\int_M\frac{4(n-1)}{n-2}|\nabla \varphi|^2dv+\int_M\mathrm{R}\varphi^2dv\right)\right]\\
&=-\frac n2\ln Y(M,[g])+\frac n2\ln\left[\int_M\left(\frac{4(n-1)}{n-2}|\nabla \varphi|^2+\mathrm{R}\varphi^2\right)dv\right].
\end{align*}
Using the elementary inequality $\ln x\leq \sigma x-(1+\ln\sigma)$ for any
real number $\sigma>0$, the above estimate can be simplified as
\[
\int_M\varphi^2\ln\varphi^2dv\le-\frac n2\ln Y(M,[g])+\frac{n\sigma}{2}\int_M
\left(\frac{4(n-1)}{n-2}|\nabla\varphi|^2+\mathrm{R}\,\varphi^2\right)dv-\frac n2(1+\ln\sigma).
\]
Setting $\tau=\frac{n\sigma}{2}$, then
\[
\int_M\varphi^2\ln \varphi^2dv\le\tau\int_M\left(\frac{4(n-1)}{n-2}|\nabla\varphi|^2+\mathrm{R}\varphi^2\right)dv
-\frac n2\ln Y(M,[g])-\frac n2\ln(2\tau)+\frac n2\ln\frac{n}{e}
\]
and hence the result follows by arranging some terms.
\end{proof}

Now we will show that Lemma \ref{logsi} indeed implies an important functional inequality
by adapting the arguments of \cite{[To2],[Wu]}. This inequality is linked with some maximal
function of scalar curvature and the volume ratio.
\begin{theorem}\label{logeq}
Let $(M,g)$ be an $n$-dimensional $(n\ge3)$ closed manifold with positive scalar curvature.
 For any point $p\in M$ and for any $r>0$,
\[
\frac n2\ln\frac{2eY(M,[g])}{n}\le\frac{16(n-1)}{n-2}\cdot\frac{V(p,r)}{V\left(p,\frac r2\right)}
+\frac{r^2}{V\left(p,\frac r2\right)}\int_{B(p,r)}\mathrm{R}dv
+\ln \frac{V(p,r)}{r^n},
\]
where $\mathrm{R}$ is the scalar curvature of $(M,g)$, $Y(M,[g])$ is the Yamabe constant
and $V(p,r)$ is the volume of geodesic ball $B(p,r)$ with radius $r$ and center $p$.
\end{theorem}

\begin{proof}[Proof of Theorem \ref{logeq}]
We choose a smooth cut-off function $\psi:[0,\infty)\to[0,1]$ supported in $[0,1]$ such that
$\psi(t)=1$ on $[0,1/2]$ and $|\psi'|\le 2$ on $[0,\infty)$. For any point
$p\in M$, we let
\[
\varphi(x):=e^{-\frac{\lambda}{2}}\psi\left(\frac{d(p,x)}{r}\right),
\]
where $\lambda$ is some constant determined by the constraint condition $\int_M\varphi^2dv=1$. Obviously,
$\lambda$ satisfies
\[
V\left(p,\frac r2\right)\le e^{\lambda}\int_M\varphi^2dv=e^{\lambda}
\]
and
\[
e^{\lambda}=e^{\lambda}\int_M\varphi^2dv=\int_M\psi^2(d(p,x)/r)dv\le V(p,r).
\]
That is, the constant $\lambda$ has upper and lower bounds as follows
\[
V\left(p,\frac r2\right)\le e^{\lambda}\le V(p,r).
\]
We now apply the above cut-off function $\varphi$ to simplify the logarithmic
Yamabe-Sobolev inequality in Lemma \ref{logsi}. Notice that $\varphi$ satisfies
\[
|\nabla\varphi|\le\frac 2r\cdot e^{-\frac{\lambda}{2}},
\]
which is supported in $B(p,r)$. Let us estimate each term of the right hand side
of \eqref{LSI}.

For the first term of the right hand side of \eqref{LSI}, we have that
\begin{equation}\label{est1}
\begin{aligned}
\frac{4(n-1)}{n-2}\tau\int_M |\nabla\varphi|^2dv&=\frac{4(n-1)}{n-2}\tau\int_{B(p,r)\backslash B(p,\frac r2)} |\nabla\varphi|^2dv\\
&\le\frac{4(n-1)}{n-2}\tau V(p,r)\frac{4}{r^2}e^{-\lambda}\\
&\le\frac{16(n-1)\tau}{(n-2)r^2}\cdot\frac{V(p,r)}{V\left(p,\frac r2\right)}.
\end{aligned}
\end{equation}
For the second term of the right hand side of \eqref{LSI}, we estimate
\begin{equation}\label{est2}
\begin{aligned}\tau\int_M\mathrm{R}\varphi^2 dv&\le\tau e^{-\lambda}\int_{B(p,r)}\mathrm{R}dv\\
&\le\frac{\tau}{V\left(p,\frac r2\right)}\int_{B(p,r)}\mathrm{R}dv.
\end{aligned}
\end{equation}
Next we will estimate the third term of the right hand side of \eqref{LSI}. We see that continuous
function $\Psi(t):=-t\ln t$ is concave with respect to $t>0$ and the Riemannian measure $dv$ is
supported in $B(p,r)$. Using the Jensen's inequality
\[
\frac{\int\Psi(\varphi^2)dv}{\int dv}\le\Psi\left(\frac{\int \varphi^2 dv}{\int dv}\right)
\]
and the definition of $\Psi$, we get that
\[
-\frac{\int_{B(p,r)}\varphi^2\ln\varphi^2dv}{\int_{B(p,r)}dv}
\leq-\frac{\int_{B(p,r)}\varphi^2dv}{\int_{B(p,r)}dv}\cdot\ln\left(\frac{\int_{B(p,r)}\varphi^2dv}{\int_{B(p,r)}dv}\right).
\]
Since $\int_{B(p,r)}\varphi^2dv=1$, the above inequality can be simplified as
\[
-\int_{B(p,r)}\varphi^2\ln\varphi^2dv\leq\ln V(p,r).
\]
By the definition of $\varphi(x)$, therefore
\begin{equation}\label{est3}
\begin{aligned}
-\int_M\varphi^2\ln\varphi^2dv&=-\int_{B(p,r)}\varphi^2\ln\varphi^2dv\\
&\le\ln V(p,r).
\end{aligned}
\end{equation}
Putting \eqref{est1}, \eqref{est2} and \eqref{est3} into \eqref{LSI}, we arrive at
\[
\frac n2\ln\frac{2eY(M,[g])}{n}\le\frac{16(n-1)\tau}{(n-2)r^2}\cdot\frac{V(p,r)}{V\left(p,\frac r2\right)}
+\frac{\tau}{V\left(p,\frac r2\right)}\int_{B(p,r)}\mathrm{R}dv
+\ln \frac{V(p,r)}{\tau^{\frac n2}}
\]
for any $\tau>0$. The conclusion then follows by letting $\tau=r^2$.
\end{proof}

\section{Maximal function and volume ratio}\label{sec3}
In this section, we will give an alternative property of uniformly lower
bounds between the maximal function of scalar curvature and the volume
ratio in a geodesic ball.

Inspired by Topping's arguments in \cite{[To2],[To3]}, on an $n$-dimensional $(n\ge3)$ Riemannian
manifold $(M,g)$, for any point $p\in M$ and $r>0$, we consider the \emph{maximal function}
\[
M f(p,r):=\sup_{s\in(0,r]}s^{-1}\left[V(p,s)\right]^{-\frac{n-3}{2}}\left(\int_{B(p,s)} |f|dv\right)^{\frac{n-1}{2}}
\]
for $f\in C^\infty(M)$, and the \emph{volume ratio}
\[
\kappa(p,r):=\frac{V(p,r)}{r^n}.
\]

With the help of Theorem \ref{logeq}, we will show that the maximal function
of scalar curvature and the volume ratio in closed manifolds with positive
scalar curvature cannot be simultaneously smaller than a fixed constant.
\begin{theorem}\label{maxirati}
Let $(M,g)$ be an $n$-dimensional $(n\ge3)$ closed manifold with positive scalar curvature.
Then there exits a constant $\delta>0$ depending only on $n$ and $Y:=Y(M,[g])$ such that for
any point $p\in M$ and for any $r>0$, at least one of the following is true:
\begin{enumerate}
 \item
$M \mathrm{R}(p,r)>\delta$;

 \item
 $\kappa(p,r)>\delta$.
\end{enumerate}
Here $\mathrm{R}(p,r)$ denotes the scalar curvature in the geodesic ball $B(p,r)$.
In particular, we can take
\[
\delta=\min\left\{w_n,\,\left(2en^{-1}Y\right)^{\frac n2}e^{-\frac{17n-18}{n-2}\cdot 2^n}\right\},
\]
where $Y:=Y(M,[g])$ is the Yamabe constant and $w_n$ is the volume of
the unit $n$-sphere $\mathbb{S}^n$.
\end{theorem}
\begin{proof}[Proof of Theorem \ref{maxirati}]
The proof is similar to the proof of Theorem 3.1 in \cite{[Wu]}. We give its detailed proof
here for the sake of completeness,. Suppose that there exist a point $p\in(M,g)$ and $r>0$ such that
$M\mathrm{R}(p,r)\le\delta$ for some constant $\delta>0$. For any $0<\epsilon<1$, we define
constant $\delta$ as follows:
\[
\delta:=\min\left\{(1-\epsilon)w_n,\,\left(2en^{-1}Y\right)^{\frac n2}e^{-\frac{17n-18}{n-2}\cdot 2^n}\right\}.
\]

In the following our aim is to prove $\kappa(p,r)>\delta$. If it is not true, we make the following

\vspace{.1in}

\textbf{Claim}. \emph{Suppose there exist a point $p\in M$ and $r>0$ such that
$M \mathrm{R}(p,r)\le\delta$ for some constant $\delta>0$. If
$\kappa(p,s)\le\delta$, then $\kappa(p,s/2)\le\delta$ for any $s\in(0,r]$.}

\vspace{.1in}

This claim will be proved later. We now continue to prove Theorem \ref{maxirati}. We repeatedly use
the claim and finally have that
\begin{equation}\label{cline}
\kappa\left(p,\frac{r}{2^m}\right)\leq\delta\le (1-\epsilon)w_n
\end{equation}
for any $m\in\mathbb{N}$, where $\epsilon$ is the sufficiently small positive constant.
But if we let $m\to\infty$, then
\[
\kappa\left(p,\frac{r}{2^m}\right)\to w_n,
\]
which contradicts \eqref{cline}. Therefore $\kappa(p,r)>\delta$ and the theorem follows.
The desired constant $\delta$ is obtained by letting $\epsilon\to 0+$.
\end{proof}

In the rest, we only need to check the above claim. We adopt the argument from \cite{[Wu]}.
\begin{proof}[Proof of Claim]
According to the relative sizes of $V(p,s/2)$ and $V(p,s)$, we may prove the claim by two cases.

\emph{Case one}. Suppose that
\[
V\left(p,\frac s2\right)\leq\delta^{\frac{2}{n-1}}2^{-n}s^{\frac{2n}{n-1}}\left[V(p,s)\right]^{\frac{n-3}{n-1}}.
\]
Then,
\begin{align*}
\kappa\left(p,\frac s2\right)&:=\frac{2^n}{s^n}V\left(p,\frac s2\right)\\
&\le\delta^{\frac{2}{n-1}}s^{\frac{2n}{n-1}-n}\left[V(p,s)\right]^{\frac{n-3}{n-1}}\\
&=\delta^{\frac{2}{n-1}}(\kappa(p,s))^{\frac{n-3}{n-1}}\\
&\leq\delta^{\frac{2}{n-1}}\delta^{\frac{n-3}{n-1}}\\
&=\delta,
\end{align*}
which proves the claim.

\emph{Case Two}. Suppose that
\[
V\left(p,\frac s2\right)>\delta^{\frac{2}{n-1}}2^{-n}s^{\frac{2n}{n-1}}\left[V(p,s)\right]^{\frac{n-3}{n-1}}.
\]
Since $M \mathrm{R}(p,r)\le\delta$ and $\mathrm{R}>0$, according to the definition
of $M \mathrm{R}(p,r)$, we indeed have
\[
\int_{B(p,s)}\mathrm{R}dv\le\delta^{\frac{2}{n-1}}s^{\frac{2}{n-1}}\left[V(p,s)\right]^{\frac{n-3}{n-1}}
\]
for all $s\in(0,r]$. Using the assumption of Case Two, the above estimate can be reduced to
\[
\int_{B(p,s)}\mathrm{R}dv\le2^n s^{-2}V\left(p,\frac s2\right)
\]
for all $s\in(0,r]$. Substituting this into Theorem \ref{logeq},
\begin{equation}\label{simifyine}
\begin{aligned}
\frac n2\ln\frac{2eY(M,[g])}{n}&\le\frac{16(n-1)}{n-2}\cdot\frac{V(p,s)}{V\left(p,\frac s2\right)}
+\frac{s^2}{V\left(p,\frac s2\right)}\int_{B(p,s)}\mathrm{R}dv+\ln \kappa(p,s)\\
&\le\frac{16(n-1)}{n-2}\cdot\frac{V(p,s)}{V\left(p,\frac s2\right)}+2^n+\ln\delta
\end{aligned}
\end{equation}
for all $s\in(0,r]$, where in the above second line we used $\kappa(p,s)\le\delta$.

On the other hand, the definition of $\delta$ implies that
\[
\ln \delta\le\frac n2\ln\frac{2eY(M,[g])}{n}-\frac{17n-18}{n-2}\cdot 2^n.
\]
Substituting this into \eqref{simifyine},
\[
\frac{V(p,s)}{V\left(p,\frac s2\right)}\ge 2^n
\]
for all $s\in(0,r]$. Therefore, for any $s\in(0,r]$, we have
\begin{align*}
\kappa\left(p,\frac s2\right)&:=\frac{2^n\cdot V\left(p,\frac s2\right)}{s^n}\\
&\le\frac{V(p,s)}{s^n}\\
&=\kappa(p,s)\\
&\le\delta
\end{align*}
and the claim follows.
\end{proof}

%%%%%%%%%%%%%%%%%%%%%%%%%%%%%%%%%%%%%%%%%%%%%%%%%%%%%%%%%%%%%%%%%%%%%%%%%%%%%%%%%%%%%%%%%%%%%%%%%%%%%%%%%%%%%%%%%%%%%%%%%%%%%%%%%%%%%%%%%%%%%%
%%%%%%%%%%%%%%%%%%%%%%%%%%%%%%%%%%%%%%%%%%%%%%%%%%%%%%%%%%%%%%%%%%%%%%%%%%%%%%%%%%%%%%%%%%%%%%%%%%%%%%%%%%%%%%%%%%%%%%%%%%%%%%%%%%%%%%%%%%%%%%

\section{Diameter estimate}\label{sec4}

In this section we will apply Theorem \ref{maxirati} to prove Theorem \ref{Main}
by adapting the arguments of \cite{[To2],[Wu]}. We need to carefully examine the
explicit coefficients of the diameter estimate in terms of the positive Yamabe
constant in \eqref{YSineq}.

To prove Theorem \ref{Main}, we will need a Vitali-type covering lemma
(see \cite{[To2]}, or \cite{[WZ]}), which is the key step to prove our theorem.

\begin{lemma}\label{cover}
Let $\gamma$ be a shortest geodesic connecting any two points $x$ and $y$ in $(M,g)$,
and $s$ be a nonnegative bounded function defined on $\gamma$. If
$\gamma\subset\{B(p,s(p))~|~p\in \gamma\}$, then for any  $\rho\in(0,\frac 12)$,
there exists a countable (possibly finite) set of points $\{p_i\in\gamma\}$ such that
\begin{enumerate}
 \item $B(p_i,s(p_i))$ are disjoint;

 \item $\gamma\subset\cup_i B(p_i,s(p_i))$;

 \item $\rho\,dist(x,y)\le \sum_i 2s(p_i)$, where $dist(p_1,p_2)$ denotes the distance
 between $x$ and $y$ in $(M,g)$.
\end{enumerate}
\end{lemma}

Now we can finish the proof of Theorem \ref{Main}.

\begin{proof}[Proof of Theorem \ref{Main}]
For the fixed constant $\delta$ defined in Theorem \ref{maxirati} and the
closed manifold $M$, we could choose $r_0>0$ sufficiently large so that
the total volume of $M$ is less than $\delta r_0^n$ because the total volume of
$M$ is finite. For any point $p\in M$, we conclude that
\[
\kappa(p,r_0)=\frac{V(p,r_0)}{r_0^n}\le\frac{V(M)}{r_0^n}\le\delta,
\]
where $V(M)$ denotes the volume of $M$. By Theorem \ref{maxirati}, it indicates
that $M \mathrm{R}(p,r_0)>\delta$. That is, there exists $s=s(p)>0$ such that
\begin{equation}\label{defMR}
\delta<s^{-1}\big[V(p,s)\big]^{-\frac{n-3}{2}}\left(\int_{B(p,s)} \mathrm{R}dv\right)^{\frac{n-1}{2}}.
\end{equation}
Notice that the following H\"older inequality holds
\[
\int_{B(p,s)} \mathrm{R}dv\le\left(\int_{B(p,s)} \mathrm{R}^{\frac{n-1}{2}}dv\right)^{\frac{2}{n-1}}\cdot\left(\int_{B(p,s)}dv\right)^{\frac{n-3}{n-1}}.
\]
Applying this, we can estimate \eqref{defMR} by
\[
\delta< s^{-1}\int_{B(p,s)} \mathrm{R}^{\frac{n-1}{2}}dv.
\]
In the other words,
\begin{equation}\label{intine}
s(p)<\delta^{-1}\int_{B(p,s(p))} \mathrm{R}^{\frac{n-1}{2}}dv.
\end{equation}

In the next step, we shall pick appropriate points $p$ such that \eqref{intine} will be
used in these points. Assume that $p_1$ and $p_2$ are two extremal points in the closed
manifold $(M,g)$ such that $\mathrm{diam}(M)=dist(p_1,p_2)$. Let $\gamma$ be a shortest
geodesic connecting $p_1$ and $p_1$. Then we obviously have
$\gamma\subset\{B(p,s(p))~|~p\in\gamma\}$. By Lemma \ref{cover}, there exists a
countable (possibly finite) set of points $\{p_i\in\gamma\}$ such that geodesic balls
$\{B(p_i,s(p_i))\}$ are disjoint and
\[
\rho\,\mathrm{diam}(M)=\rho\,dist(p_1,p_2)\le \sum_i 2s(p_i).
\]
Substituting \eqref{intine} into the above inequality,
\begin{equation}
\begin{aligned}\label{precisein}
\mathrm{diam}(M)&\leq\frac{2}{\rho}\sum_is(p_i)\\
&<\frac{2}{\rho}\delta^{-1}\sum_i\int_{B(p_i,s(p_i))}\mathrm{R}^{\frac{n-1}{2}}dv\\
&\leq \frac{2}{\rho}\delta^{-1}\int_M\mathrm{R}^{\frac{n-1}{2}}dv,
\end{aligned}
\end{equation}
where $\delta>0$ is a constant depending only on $n$ and $Y(M,[g])$.
Letting $\rho\to\frac{1}{2}-$, we have
\[
\mathrm{diam}(M)\le4\delta^{-1}\int_M\mathrm{R}^{\frac{n-1}{2}}dv,
\]
where $\delta$ is chosen as in Theorem \ref{maxirati}.
The desired estimate follows.
\end{proof}

In the course of proving Theorem \ref{Main}, we indeed prove a quantitative
estimate for the diameter of some collapsed geodesic ball. It
states that if there exist a point $p\in (M,g)$ and a real number
$r_0>0$ such that
\[
Y=Y(M,[g])>0
\]
in $B(p,2r_0)$ and
\[
\frac{V(p,r_0)}{r_0^n}<\delta,
\]
where $\delta:=\min\{w_n,\,(2en^{-1}Y)^{\frac n2}e^{\frac{18-17n}{n-2}\cdot 2^n}\}$,
then
\[
\mathrm{diam}(B(p,r_0))\le 4\delta^{-1}\int_{B(p, 2r_0)}\mathrm{R}^{\frac{n-1}{2}}.
\]

\section{Proof of Theorem \ref{Mainimm}}\label{sec5}
When a closed manifold has a conformal immersion into a sphere, we can
prove another upper diameter bound (i.e. Theorem \ref{Mainimm}) which
depends only on the positive part of scalar curvature integral.

The proof of Theorem \ref{Mainimm} is almost the same as the argument of Theorem
\ref{Main}. Indeed, under the assumption of Theorem \ref{Mainimm}, we have the
following Sobolev inequality which does not depend on the Yamabe constant; see
Proposition 3.25 in \cite{[He]}.
\begin{proposition}\label{immerpro}
Let $(M,g)$ be an $n$-dimensional $(n\ge3)$ simply connected closed manifold.
Assume that there exists a conformal immersion
\[
\Psi:\,(M,g)\to (\mathbb{S}^n,g_0),
\]
where $(\mathbb{S}^n,g_0)$ denotes the standard unit sphere in $\mathbb{R}^{n+1}$.
For any positive $\varphi\in C^\infty(M)$,
\[
\left(\int_M\varphi^{\frac{2n}{n-2}}dv\right)^{\frac{n-2}{n}}
\le \left[n(n-1)w_n^{\frac 2n}\right]^{-1}\left(\int_M\frac{4(n-1)}{n-2}|\nabla \varphi|^2dv+\int_M\mathrm{R}_{+}\,\varphi^2dv\right),
\]
where $w_n$ is the volume of the unit $n$-sphere $(\mathbb{S}^n,g_0)$ and
$\mathrm{R}_{+}$ is the positive part of the scalar curvature $\mathrm{R}$.
\end{proposition}
With the help of Proposition \ref{immerpro}, Theorem \ref{Mainimm} easily
follows by adapting the lines of proving Theorem \ref{Main}. We omit the
repeated discussion.

\end{document}